\documentclass[a4paper,12pt,reqno]{amsart}

\usepackage{t1enc}
\usepackage[english]{babel}
\usepackage{amsmath,amssymb,eucal,amsthm}
\usepackage{bm}
\usepackage{graphicx}
\usepackage{mathrsfs}
\usepackage{mathptmx}
\usepackage{latexsym}
\usepackage{ulem}

\newtheorem{theorem}{Theorem}
\newtheorem{lemma}[theorem]{Lemma}

\theoremstyle{definition}
\newtheorem*{remark}{Remark}

\def\pp{\mathbb{P}}
\def\nn{\mathbb{N}}
\def\zz{\mathbb{Z}}

\def\cc{\mathbb{C}}
\def\qq{\mathbb{Q}}
\def\ii{\mathbb{I}}

\def\gb{\mathfrak{B}}

\def\vf{\varphi}
\def\ve{\varepsilon}
\def\vph{{\varphi}_h}

\def\d{{\rm d}}
\def\meas{{\rm meas}}

\def\uH{{\underline H}}
\def\uZ{{\underline Z}}
\def\us{\underline{s}}
\def\uO{{\underline \Omega}}
\def\uA{{\underline \alpha}}
\def\uom{{\underline \omega}}
\def\ugb{{\underline \gb}}

\def\ho{{\widehat\omega}}
\def\hP{{\widehat P}}

\def\st{{\widetilde{S}}}
\def\tvf{{\widetilde{\varphi}}}

\def\no{{\mathbb{N}_0}}

\begin{document}
\hfill\texttt{\jobname.tex}\qquad\today

\bigskip
\title[Discrete case of mixed joint universality]
{The discrete case of the mixed joint universality for a class of certain
partial zeta-functions}

\author{Roma Ka{\v c}inskait{\.e}}

\address{R. Ka{\v c}inskait{\.e} \\
Department of Mathematics and Statistics, Vytautas Magnus University, Vileikos 8, Kaunas LT-44404, Lithuania}
\email{roma.kacinskaite@vdu.lt}

\author{Kohji Matsumoto}

\address{K. Matsumoto, Graduate School of Mathematics, Nagoya University, Chikusa-ku,
Nagoya 464-8602, Japan}
\email{kohjimat@math.nagoya-u.ac.jp}
\date{}

\begin{abstract}
We give a new type of mixed discrete joint universality pro\-per\-ties, which is satisfied by a wide class of zeta-functions. We study the universality for a certain mo\-di\-fication of Matsumoto zeta-functions $\vph(s)$ and the periodic Hurwitz zeta-function $\zeta(s,\alpha;\gb)$ under the condition that the common difference of arithmetical progression $h>0$ is such that $\exp\{\frac{2 \pi}{h}\}$ is a rational number and parameter $\alpha$ is a transcendental number. Also we generalize a known discrete universality result in \cite{RK-KM-2017-Pal} and the new one as well.
\end{abstract}

\maketitle

{\small{Keywords: {approximation, discrete shift, Euler products, periodic Hurwitz zeta-function,
Matsumoto zeta-function, value distribution, universality.}}}

{\small{AMS classification:} 11M06, 11M41, 11M36.}

\section{Introduction}\label{sec-1}

In 2015, the first result on the mixed joint universality theorem for a general polynomial Euler product (or so-called Matsumoto zeta-function) $\varphi(s)$ belonging to the Steuding class $\widetilde S$
and a periodic Hurwitz zeta-function $\zeta(s,\alpha;\gb)$ was obtained by the authors (see \cite{RK-KM-2015}). In 2017, this result was generalized to the case of the tuple consisting of
one Matsumoto zeta-function and several periodic Hurwitz zeta-functions (see \cite{RK-KM-2017-BAMS}).

We recall the definitions of both of the above functions.   Let $s=\sigma+it$ be a complex variable, and by $\mathbb{P}$, $\mathbb{N}$, $\mathbb{N}_0$, $\mathbb{Z}$, $\mathbb{Q}$ and $\mathbb{C}$
denote the sets of all primes, positive integers, non-negative integers, integers, rational numbers and complex numbers, respectively. Let $\gb=\{b_m: m \in \no\}$ be a periodic
sequence of complex numbers $b_m$ with a minimal period $k \in \mathbb{N}$,
 and suppose that $\alpha$ is a fixed real number, $0<\alpha \leq 1$. Then, for $\sigma>1$, the periodic Hurwitz zeta-function is defined by the Dirichet series
$$
\zeta(s,\alpha;\gb)=\sum_{m=0}^{\infty}\frac{b_m}{(m+\alpha)^s}.
$$
For $\sigma>1$, the function $\zeta(s,\alpha;\gb)$ can be expressed as a linear combination of classical Hurwitz zeta-functions $\zeta(s,\alpha)$  (see \cite{AL-2006}), i.e.,
$$
\zeta(s,\alpha;\gb)=\frac{1}{k^s}\sum_{l=0}^{k-1}b_l\zeta\left(s,\frac{l+\alpha}{k}\right),
$$
from which we deduce that it can be analytically continued to the whole $s$-plane except for a possible simple pole at the point $s=1$ with residue
$
b:=k^{-1}(b_0+\cdots+b_{k-1}).
$

The polynomial Euler products $\tvf(s)$ or so-called Matsumoto zeta-functions are given  by the formula
\begin{equation}\label{rk-fo-1}
\tvf(s)=\prod_{m=1}^{\infty}\prod_{j=1}^{g(m)}\left(1-a_m^{(j)}p_m^{-sf(j,m)}\right)^{-1}
\end{equation}
for $m \in \mathbb{N}$, $g(m)\in \mathbb{N}$, $j \in \mathbb{N}$, $1 \leq j \leq g(m)$, $f(j,m)\in \mathbb{N}$, and $m$th prime number  $p_m$ (see \cite{KM-1990}). Suppose that, for non-negative constants
$\alpha$ and $\beta$, the inequalities
\begin{equation}\label{rk-fo-2}
g(m)\leq C_1 p_m^\alpha \quad \text{and} \quad |a_m^{(j)}|\leq p_m^\beta
\end{equation}
hold with a positive constant $C_1$. In view of this assumption, the right-side of the equality \eqref{rk-fo-1} converges absolutely for $\sigma>\alpha+\beta+1$, and in this half-plane the function $\tvf(s)$ can be presented by the Dirichlet series
$$
\tvf(s)=\sum_{k=1}^{\infty}\frac{{\widetilde c}_k}{k^s},
$$
where the coefficients ${\widetilde c}_k$ satisfy an estimate ${\widetilde c}_k=O(k^{\alpha+\beta+\varepsilon})$ with every postive $\varepsilon$ if all prime factors of $k$ are large
(for the comments, see Appendix in \cite{RK-KM-2017-BAMS}). For brevity, denote the shifted version of $\tvf(s)$ by
\begin{equation}\label{fo-3}
 \vf(s):=\tvf(s+\alpha+\beta)=\sum_{k=1}^{\infty}\frac{c_k}{k^s},
\end{equation}
where $c_k:=k^{-\alpha-\beta}{{\widetilde c}_k}$. Then $\vf(s)$ is an absolutely convergent series for $\sigma>1$. Also, let the function $\varphi(s)$ be such that:
\begin{itemize}
  \item[(i)] it can be continued meromorphically to
$\sigma\geq\sigma_0$, $\frac{1}{2}\leq\sigma_0<1$, and all poles in this
region are included in a compact set which has no intersection with the line $\sigma=\sigma_0$,
  \item[(ii)] for $\sigma\geq\sigma_0$,
  $
  \varphi(\sigma+it)=O(|t|^{C_2})
  $
   holds with a positive constant $C_2$,
  \item[(iii)] it holds the mean-value estimate
\begin{equation}\label{rk-eq-2-5}
\int_0^T|\varphi(\sigma_0+it)|^2 dt=O(T), \quad T \to \infty.
\end{equation}
\end{itemize}
We denote the set of all such $\varphi(s)$ by $\mathcal{M}$.

Now we recall the definition of the Steuding class $\st$ (see \cite{JSt-2007}). We say that the function $\vf(s)$ belongs to this class if following conditions are fulfilled:
\begin{itemize}
  \item[(a)] there exists a Dirichlet series expansion
  $$
  \varphi(s)=\sum_{m=1}^{\infty}\frac{a(m)}{m^s}
  $$
  with $a(m)=O(m^\varepsilon)$ for every $\varepsilon>0$;
  \item[(b)] there exists $\sigma_\varphi<1$ such that $\varphi(s)$ can be meromorphically continued to the half-plane $\sigma>\sigma_\varphi$;
  \item[(c)] for every fixed $\sigma>\sigma_\varphi$ and $\varepsilon>0$, there exists a constant $C_3 \geq 0$ such that
  $$
  \varphi(\sigma+it)=O(|t|^{C_3+\varepsilon});
  $$
  \item[(d)] there exists the Euler product expansion over primes, i.e.,
  $$
  \varphi(s)=\prod_{p \in \mathbb{P}}\prod_{j=1}^{l}\left(1-\frac{a_j(p)}{p^s}\right)^{-1};
  $$
  \item[(e)] there exists a constant $\kappa>0$ such that
  $$
  \lim_{x \to \infty}\frac{1}{\pi(x)}\sum_{p \leq x}|a(p)|^2=\kappa,
  $$
  where $\pi(x)$ denotes the number of primes $p$ not exceeding $x$.
\end{itemize}
Denote by $\sigma^*$ the infimum of all $\sigma_1$ such that
$$
\frac{1}{2T}\int_{-T}^{T}|\vf(\sigma+it)|^2 \d t \sim \sum_{m=1}^{\infty}\frac{|a(m)|^2}{m^{2\sigma}}
$$
holds for any $\sigma \geq \sigma_1$.
Then $\frac{1}{2}\leq \sigma^*<1$.
This implies that $\st \subset\mathcal{M}$.

Also, throughout this paper we will use the following notation and definitions. By $H(G)$ we denote the space of holomorphic functions on a region $G$
with the uniform convergence topology (here $G$ is any open region in the complex plane).
Let $K \subset \mathbb{C}$ be a compact set. Denote by $H^c(K)$
the set of all $\mathbb{C}$-valued continuous  functions on $K$ and holomorphic in the interior of $K$, and by $H_0^c(K)$ the subset of
elements of $H^c(K)$ which are non-zero on $K$, respectively. Let $D(a,b)=\{s \in \mathbb{C}: a <\sigma <b\}$ for every $a<b$, and denote
by $\meas\{A\}$ the Lebesgue measure of the measurable set $A \subset \mathbb{R}$. Denote by ${\mathcal B}(S)$ the set of all Borel subsets of a topological space $S$.

Now we recall the statement of our first result on the mixed joint universality property, which is of continuous character, for the functions $\vf(s)$ and $\zeta(s,\alpha;\gb)$.
This result is Theorem~2.2 in \cite{RK-KM-2015} (while a more general case is contained in Theorem~4.2 in \cite{RK-KM-2017-BAMS}).

\begin{theorem}[\cite{RK-KM-2015}]\label{rk-th-1}
Suppose that $\vf(s)$ belongs to the Steuding class $\st$, and $\alpha$ is a trans\-cendental number. Let $K_1$ be a compact subset of $D(\sigma^*,1)$, $K_{2}$ be a
compact subset of $D\big(\frac{1}{2},1\big)$, both with connected complements. Then, for any $f_1 \in H_0^c(K_1)$, $f_{2}\in H^c(K_{2})$ and every $\ve>0$, it holds that
\begin{eqnarray*}
\liminf\limits_{T \to \infty}\frac{1}{T}\meas \bigg\{\tau\in [0,T]: && \sup\limits_{s \in K_1}|\varphi(s+i\tau)-f_1(s)|<\varepsilon, \\
&& \sup\limits_{s\in K_{2}}|\zeta(s+i\tau,\alpha;\gb)-f_{2}(s)|<\varepsilon\bigg\}>0.
\end{eqnarray*}
\end{theorem}

This theorem shows that the set of shifts $\tau$, with which the pair
$\big(\vf(s+i\tau),\zeta(s+i\tau,\alpha;\gb)\big)$ approximates the tuple of holomorphic functions $\big(f_1(s),f_2(s)\big)$,
is sufficiently rich and has a positive lower density.

The mixed joint universality property of discrete character is sometimes more in\-te\-res\-ting.
In this case, we study an approximation of the functions when the imaginary part of the complex variable $s$ varies only on
the values from a certain arithmetic progression with a common difference $h>0$.

In 2017, the discrete mixed joint universality for the pair $\big(\vf(s),\zeta(s,\alpha;\gb)\big)$ was proved by the authors (see \cite{RK-KM-2017-Pal}) under a condition for the set
$$
L({\mathbb{P}},\alpha,h):=\bigg\{\log p: p \in {\mathbb{P}}\bigg\} \cup \bigg\{\log(m+\alpha): m \in \nn_0 \bigg\}
\cup\bigg\{\frac{2 \pi}{h}\bigg\}.
$$

\begin{theorem}[\cite{RK-KM-2017-Pal}]\label{rk-th-2}
Let $\varphi(s)$, $K_1$, $K_2$, $f_1(s)$ and $f_2(s)$ be as in Theorem~\ref{rk-th-1}. Suppose that the elements of the set $L(\pp,\alpha,h)$ are linearly independent over $\qq$.
Then, for every $\varepsilon>0$,
\begin{eqnarray*}
\liminf\limits_{N \to \infty}
\frac{1}{N+1}
\#
\bigg\{0\leq k \leq N:
 && \sup\limits_{s \in K_1}|\varphi(s+ikh)-f_1(s)|<\varepsilon, \\ &&  \sup\limits_{s\in K_2}|\zeta(s+ikh,\alpha;\gb)-f_2(s)|<\varepsilon\bigg\}>0.
\end{eqnarray*}
\end{theorem}

\section{Statements of results}\label{sec-1-1}

The main aim of this paper is to give a new type of mixed discrete joint universality theorem for the aforementioned functions under a different condition from that in Theo\-rem~\ref{rk-th-2}.
Suppose that $\exp\{\frac{2 \pi}{h}\}\in\mathbb{Q}$. Then we can write $\exp\{\frac{2 \pi}{h}\}=\frac{a}{b}$ for $(a,b)=1$, $a, b \in \zz$.   Write the factorization of numbers $a$ and $b$
into primes as $a=q_{1}^{\alpha_{1}}...q_{d(1)}^{\alpha_{d(1)}}$ and $b=r_1^{\beta_1}...r_{d(2)}^{\beta_{d(2)}}$, respectively. Put $\mathbb{P}_p=\{q_1,...,q_{d(1)},r_1,...,r_{d(2)}\}$, and let $\pp_h:=\pp \setminus \pp_p$. Denote
the set of all $m \in \nn$ such that $p_m$ (the $m$th prime) $\in \pp_p$ by $\nn_p$, and let
$\nn_h:=\nn \setminus \nn_p$.    Also denote the set of all $m \in \nn$ whose all prime divisors belong to $\pp_h$ by $\nn_m$.

Under the above notation, for $\sigma>\alpha+\beta+1$, we define a modification of the Matsumoto zeta-function ${\tvf}(s)$ by the formula
\begin{eqnarray}\label{rk-fo-4}
{\widetilde{\varphi}}_h(s)=\prod_{m\in\mathbb{N}_h}\prod_{j=1}^{g(m)}\left(1-a_m^{(j)}p_m^{-sf(j,m)}\right)^{-1},
\end{eqnarray}
and by $\vph(s)$ we denote its shifted version, i.e., $\vph(s):={\widetilde \varphi}_h(s+\alpha+\beta)$.
We call ${\widetilde{\varphi}}_h(s)$ and $\vph(s)$ the partial Matsumoto zeta-functions.
Note that the dif\-fe\-ren\-ce between $\varphi_h(s)$ and $\varphi(s)$ is only finitely many Euler factors. Therefore the function $\varphi_h(s)$ satisfies the properties (i), (ii) and (iii) too, so $\varphi_h\in\mathcal{M}$.
Moreover, if $\varphi\in\widetilde{S}$, then $\varphi_h\in\widetilde{S}$.

Now we are ready to give the statement of the main result in this paper.
We note that only the statement of this theorem was announced in \cite{RK-KM-2017-Pal}.

\begin{theorem}\label{rk-th-3}
Suppose that $\alpha$ is transcendental, $h>0$, and $\exp\big\{\frac{2\pi}{h}\big\}$ is a rational number. Let $\varphi(s)\in \st$ (and so $\varphi_h(s)\in\st$).
Suppose $K_1$, $K_2$, $f_1(s)$ and $f_2(s)$ satisfy the conditions of Theorem~\ref{rk-th-1}.
Then, for every $\varepsilon>0$, it holds that
\begin{align*}
\liminf\limits_{N \to \infty}
\frac{1}{N+1}
\#
\bigg\{0\leq k \leq N:
 & \sup\limits_{s \in K_1}|\varphi_h(s+ikh)-f_1(s)|<\varepsilon, \\ &  \sup\limits_{s\in K_2}|\zeta(s+ikh,\alpha;\gb)-f_2(s)|<\varepsilon\bigg\}>0.
\end{align*}
\end{theorem}

\begin{remark}
A typical case when the elements of $L({\mathbb{P}},\alpha,h)$ are linearly independent
is that $\alpha$ and $\exp\big\{\frac{2\pi}{h}\big\}$ are algebraically independent over $\qq$ (see
\cite{EB-AL-2015rj}).
On the other hand, in Theorem \ref{rk-th-3} we assume that $\exp\big\{\frac{2\pi}{h}\big\}$ is rational.
Therefore the arithmetic nature of $h$ in Theorem \ref{rk-th-3} is quite different from
that in Theorem \ref{rk-th-2}.
\end{remark}

The next two results are generalizations of Theorems~\ref{rk-th-2} and \ref{rk-th-3}, which are extended to the case of a collection of  periodic Hurwitz zeta-functions.
Define
$$
L({\mathbb{P}},\uA,h):=\bigg\{\log p: p \in {\mathbb{P}}\bigg\} \cup \bigg\{\log(m+\alpha_j): m \in \nn_0, \ j=1,...,r\bigg\} \cup\bigg\{\frac{ \pi}{h}\bigg\}.
$$	

\begin{theorem}\label{rk-th-2-1}
	Let $\varphi(s) \in \st$,  $K_1$ be a compact subset of $D(\sigma^*,1)$, $K_{2j}$  be compact subset of
	$D\big(\frac{1}{2},1\big)$, $j=1,...,r$, all of them with connected complements, and $f_1\in H_0^c(K_1)$, $f_{2j}\in H^c(K_{2j})$.
	Suppose that the elements of the set $L(\pp,\uA,h)$ are linearly independent over $\qq$. Then, for every $\ve>0$,
		\begin{eqnarray}\label{rk-fo-5-0}
			\liminf\limits_{N \to \infty}
			\frac{1}{N+1}
			\#
			\bigg\{0\leq k \leq N:
			&& \sup\limits_{s \in K_1}|\varphi(s+ikh)-f_1(s)|<\varepsilon, \cr &&  \sup\limits_{1\leq j \leq r}\sup\limits_{s\in K_{2j}}|\zeta(s+ikh,\alpha_j;\gb_j)-f_{2j}(s)|<\varepsilon\bigg\}>0.
		\end{eqnarray}
\end{theorem}

The second generalization deals with the partial Matsumoto zeta-function $\vph(s)$ and a collection of periodic Hurwitz zeta-functions.

\begin{theorem}\label{rk-th-3-1}
	Suppose that  $\varphi(s)$ (and so $\varphi_h(s)$) and  $\exp\{\frac{2 \pi}{h}\}$ are as in Theorem~\ref{rk-th-3}, and the numbers $\alpha_1,...,\alpha_r$ are algebraically independent over $\qq$.	Let $K_1$, $f_1(s)$,  $K_{2j}$ and $f_{2j}(s)$, $j=1,...,r$, be as in Theorem~\ref{rk-th-2-1}. Then the universality inequality of the form \eqref{rk-fo-5-0} holds when $\vf(s)$ is replaced by $\vph(s)$.
\end{theorem}

\begin{remark}
The mixed joint universality property  for the  tuple of different types of zeta-functions (one having an Euler product expression and the other without them) was introduced  by H.~Mishou in 2007 (see \cite{HM-2007}), and independently by J.~Steuding and J.~Sanders in 2006 (see \cite{JSt-JS-2006}). They obtained that any two holomorphic functions can be approximated simul\-ta\-neous\-ly by shifts of the Riemann zeta-function $\zeta(s)$ and the Hurwitz zeta-function $\zeta(s,\alpha)$.

The concept of discrete universality was introduced by A.~Reich studying the Dedekind zeta-functions in 1980 (see \cite{AR-1980}).
Mixed discrete joint universality theorems for zeta-functions are interesting and complicated objects for the investigation, because an important role is played by arithmetic properties of parameters occurring in the theorems.
Until this moment only few papers related to this problem have appeared. Two papers by E.~Buivydas and A.~Laurin\v cikas studied the mixed discrete universality for the collection $(\zeta(s),\zeta(s,\alpha))$ (see \cite{EB-AL-2015rj}, \cite{EB-AL-2015lmj}).
The authors' result in \cite{RK-KM-2017-Pal}, which is stated as Theorem~\ref{rk-th-2}, gives a generalization of the result in \cite{EB-AL-2015rj}.
Note that the first attempt to prove mixed discrete joint universality theorem was made by the first author (see \cite{RK-2009}).    The proof in \cite{RK-2009} unfortunately contains incompleteness, but this work is the origin of our present investigation.
In fact, Theorem \ref{rk-th-3} is a generalization of the ``corrected'' version of \cite{RK-2009}, as was mentioned in \cite{RK-KM-2017-Pal}.
\end{remark}

The present paper is organized in the following way. The main result of the paper is mixed discrete joint universality theorem for the pair $(\vph(s), \zeta(s,\alpha;\gb))$ (Theorem~\ref{rk-th-3}), and we separate its proof into two parts:  in Section~\ref{sec-2} we prove a discrete functional limit theorem which contains the most part of novelty of this paper, and in Section~\ref{sec-3} we study the support of the limit measure and give a proof of Theorem~\ref{rk-th-3}. In the last two sections, we briefly outline the proof of  Theorems~\ref{rk-th-2-1} and \ref{rk-th-3-1}.

\section{Functional discrete limit theorem}\label{sec-2}

In this section, we assume $\varphi\in\mathcal{M}$ (and so $\varphi_h \in \mathcal{M}$), and we give a proof of a mixed joint discrete limit theorem in the sense of weakly convergent probability measures in the space of holomorphic functions for the pair $\big(\varphi_h(s),\zeta(s,\alpha;\gb)\big)$.

The function $\varphi(s)$ has only finitely many poles, say $s_1(\varphi),\ldots,s_l(\varphi)$, and let
$$
D_{\varphi}:=\{s \in \cc : \;\sigma>\sigma_0,\; \sigma\neq\Re s_j(\varphi), \;1\leq j\leq l\}.
$$
The poles of $\varphi(s)$ and of $\varphi_h(s)$ in the region $\sigma>\sigma_0$ exactly
coincide, and hence
the functions $\varphi_h(s)$ and $\varphi_h(s+ikh)$ are holomorphic in
$D_{\varphi}$.
Also, the functions $\zeta(s,\alpha;\gb)$ and $\zeta(s+ikh,\alpha;\gb)$ are holomorphic in
$$
D_{\zeta}:=
\begin{cases}
\big\{s \in \cc: \; \sigma>\frac{1}{2}\big\} & \text{if}\quad \zeta(s,\alpha;\gb)\;\; \text{is entire},\cr
\big\{s \in \cc :\;\sigma>\frac{1}{2},\; \sigma\neq 1\big\} & \text{if}\quad s=1\;\; \text{is a pole of}\;\; \zeta(s,\alpha;\gb)
\end{cases}
$$
(for the arguments, see \cite{RK-KM-2015}).

Suppose that $N>0$, $D_1$ is an open subset of $D_{\varphi}$ and $D_2$ is an open subset of
$D_{\zeta}$. For brevity, write $\uH=H(D_1)\times H(D_2)$.  On ${\uH}$, define
$$
P_{Nh}(A)=\frac{1}{N+1}\#\big\{0 \leq k \leq N: \;
\underline{Z}_h(\underline{s}+ikh)
\in A\big\}, \quad A\in\mathcal{B}(\uH),
$$
with $\underline{s}+ikh=(s_1+ikh,s_2+ikh)$, $s_1\in D_1$, $s_2\in D_2$, and
$$
\underline{Z}_h(\underline{s}):=
\big(\varphi_h(s_1),\zeta(s_2,\alpha;\gb)\big).
$$

Let $\gamma$ be the unit circle on the complex plane $\cc$, i.e., $\gamma:=\{s \in \cc: |s|=1\}$. Define three tori
$$
\Omega_1:=\prod_{p \in \pp} \gamma_p, \quad
\Omega_{1h}:=\prod_{p \in \pp_h} \gamma_{ph} \quad \text{and} \quad  \Omega_2:=\prod_{m=0}^{\infty}\gamma_m
$$
with $\gamma_p=\gamma$ for all  $p \in \pp$, $\gamma_{ph}=\gamma$ for all  $p \in \pp_h$ and $\gamma_m=\gamma$ for all $m\in \no$, respectively.   Further define
$
\Omega_h=\Omega_{1h}\times \Omega_2.
$
By the construction, the torus $\Omega_{1h}$ is a closed  subgroup of $\Omega_1$.  By the Tikhonov theorem, all the tori $\Omega_1$, $\Omega_{1h}$ and $\Omega_2$ are compact topological Abelian groups (see Lemma 5.1.5 from \cite{AL-1996}) and the torus $\Omega_h$ is also a compact topological Abelian group.
Then, on $(\Omega_h,{\mathcal B }(\Omega_h))$, there exists a probability Haar measure $m^h_H$. Here $m^h_H:=m_{H1h}\times m_{H2}$ with the Haar measures $m_{H1h}$ and $m_{H2}$ defined
on the spaces $(\Omega_{1h}, {\mathcal{B}}(\Omega_{1h}))$ and $(\Omega_2,{\mathcal B}(\Omega_2))$, respectively.   This leads to the probability space $(\Omega_h,{\mathcal B}(\Omega_h),m^h_H)$.
For the elements $\omega_1\in\Omega_1$, $\omega_{1h}\in \Omega_{1h}$ and $\omega_2 \in \Omega_2$, denote by $\omega_1(p)$, $\omega_{1h}(p)$ and $\omega_2(m)$ the projections to the coordinate spaces $\gamma_p$ for $p \in \pp$,
$\gamma_{ph}$ for  $p \in \pp_h$, and $\gamma_m$ for $m \in \no$, respectively.

For $s_1\in\cc$ and $\omega_{1h}\in\Omega_{1h}$, define
\begin{equation}\label{rk-fo-5-1}
\varphi_h(s_1,\omega_{1h})=\sum_{k\in \nn_m}\frac{c_k \omega_{1h}(k)}{k^{s_1}}=
\prod_{k \in \nn_h}\prod_{j=1}^{g(k)}\bigg(1-a_{k}^{(j)}\omega_{1h}(p_k)^{f(j,k)}p_k^{-(s_1+\alpha+\beta) f(j,k)}\bigg)^{-1}.
\end{equation}
This series, for all $\omega_{1h} \in \Omega_{1h}$, converges absolutely for $\Re s_1>1$. From the properties of Dirichlet series it follows that this series converges uniformly almost surely on any compact subsets of  $D_1$. Therefore, $\varphi_h(s_1,\omega_{1h})$ is an $H(D_1)$-valued random element defined on the probability space $(\Omega_{1h},{\mathcal B}(\Omega_{1h}),m_{H1h})$.

On $(\Omega_{2},{\mathcal B}(\Omega_{2}),m_{H2})$,    define an $H(D_2)$-valued random element $\zeta(s_2,\alpha,\omega_2;\gb)$ by
$$
\zeta(s_2,\alpha,\omega_2;\gb)=\sum_{m \in \no}
\frac{b_m\omega_2(m)}{(m+\alpha)^{s_2}}
$$
for $\omega_2 \in \Omega_2$ (for the details, see \cite{AL-2006}).

Now, for $\omega_h=(\omega_{1h},\omega_2)\in \Omega_h$ and $\underline{s}=(s_1,s_2) \in D_1\times D_2$, on $(\Omega_h,{\mathcal B }(\Omega_h),m^h_H)$, define an $\uH$-valued random element
$$
\underline{Z}_h(\underline{s},\omega_h)=(\varphi_h(s_1,\omega_{1h}),\zeta(s_2,\alpha,\omega_2;\gb)),
$$
and let $P_{{\uZ}h}$ be the distribution of this element, i.e.,
$$
P_{{\uZ}h}(A)=m^h_H\big\{\omega_h\in\Omega_h :\; \underline{Z}_h(\underline{s},\omega_h)\in A\big\},
\quad A\in\mathcal{B}(\uH).
$$

In this section, we prove the following theorem.

\begin{theorem}\label{rk-th-4}
Let $\varphi\in\mathcal{M}$ (and so $\varphi_h\in\mathcal{M}$).
Suppose that $\alpha$ is transcendental, $h>0$, and $\exp\big\{\frac{2 \pi}{h}\big\}$ is rational. Then $P_{Nh}$ converges weakly to $P_{\uZ h}$ as $N \to \infty$.
\end{theorem}


Before starting the proof of Theorem \ref{rk-th-4}, we should remark that the method developed in \cite{RK-KM-2017-Pal} can be applied in the case of $\varphi_h(s)$ too. Therefore we skip some standard details,
and focus onto the points which are essential in the proof.

We begin the proof of Theorem~\ref{rk-th-4} with a mixed joint discrete limit theorem on torus $\Omega_h$.   Define the measure $Q_{Nh}$ on $(\Omega_h, {\mathcal B}(\Omega_h))$ by
$$
Q_{Nh}(A):=\frac{1}{N+1}\# \bigg\{0 \leq k \leq N: \big(\big(p^{-ikh}: p \in \pp_h\big), \big((m+\alpha)^{-ikh}: m \in \no \big)\big)\in A\bigg\},
$$
$A \in {\mathcal B}(\Omega_h)$.

\begin{lemma}\label{rk-le-1}
Suppose that the hypotheses of Theorem~\ref{rk-th-4} are satisfied. Then $Q_{Nh}$ converges weakly to the Haar measure $m^h_H$ as $N \to \infty$.
\end{lemma}

\begin{proof}
Let $(\underline{k},\underline{l}):=\big((k_p: p \in \pp_h),(l_m: m\in \no)\big)$, where only a finite number of integers $k_p$ and $l_m$ are distinct from zero.
We apply the Fourier transform method (for the details, see \cite{AL-1996}).
In view of the definition of the measure $Q_{Nh}$, the Fourier transform $g_{Nh}$ of the measure
$Q_{Nh}$ is of the form
\begin{eqnarray}\label{rk-fo-5}
g_{Nh}(\underline{k},\underline{l})&=&\int_{\Omega}\bigg(\prod_{p \in \pp_h}\omega_{1h}^{k_p}(p)\prod_{m \in \no}\omega_2^{l_m}(m)\bigg)\d Q_{N,h}\cr
&=&\frac{1}{N+1}\sum_{k=0}^{N}\prod_{p \in \pp_h}p^{-ikk_ph}\prod_{m \in \no}(m+\alpha)^{-ikl_mh}\cr
&=&
\frac{1}{N+1}\sum_{k=0}^{N}
\exp\bigg\{-ikh\bigg(\sum_{p \in \pp_h}k_p\log p+\sum_{m \in \no}l_m \log(m+\alpha)\bigg)\bigg\}.
\end{eqnarray}
Clearly, for $(\underline{k},\underline{l})=(\underline 0, \underline 0)$,
\begin{equation}\label{rk-fo-5-5}
g_{Nh}(\underline{k},\underline{l})=1.
\end{equation}

Since $\alpha$ is transcendental and $\exp\big\{\frac{2 \pi}{h}\big\}$ is rational, we have
\begin{equation}\label{rk-fo-6}
\exp\bigg(-ih\bigg(\sum_{p \in \pp_h}k_p\log p+\sum_{m \in \no}l_m \log(m+\alpha)\bigg)\bigg)\not =1
\end{equation}
for $(\underline{k},\underline{l})\not =(\underline 0, \underline 0)$.
In fact, if \eqref{rk-fo-6} is not true, then, for some integer $r$,
$$
-ih\bigg(\sum_{p \in \pp_h}k_p\log p+\sum_{m \in \no}l_m \log(m+\alpha)\bigg)=2 \pi i r.
$$
Then
$$
\sum_{p \in \pp_h}k_p\log p+\sum_{m \in \no}l_m \log(m+\alpha)=-\frac{2 \pi r}{h},
$$
and, taking the exponentials, we obtain
\begin{equation}\label{rk-fo-7}
\prod_{p \in \pp_h}p^{k_p}\prod_{m \in \no}(m+\alpha)^{l_m} =\exp\bigg(-\frac{2 \pi r}{h}\bigg).
\end{equation}
Since the right-hand side is rational, if some $l_m\neq 0$, then \eqref{rk-fo-7}
contradicts with the assumption that $\alpha$ is transcendental.
Therefore all $l_m=0$, and \eqref{rk-fo-7} reduces to
$$
\prod_{p \in \pp_h}p^{k_p}=\exp\bigg(-\frac{2 \pi r}{h}\bigg).
$$
But this is impossible
in view of the definition of $\pp_h$.    Therefore \eqref{rk-fo-6} is valid.

Therefore, in the case $(\underline{k},\underline{l})\not =(\underline 0, \underline 0)$
from \eqref{rk-fo-5}, we obtain that
\begin{eqnarray}\label{rk-fo-8}
&&g_{Nh}(\underline{k},\underline{l})\cr
&& \ =
\frac{1}{N+1}\frac{1-\exp\big\{-i(N+1)h\big(\sum_{p \in \pp_h}k_p\log p+ \sum_{m\in \no}l_m\log(m+\alpha)\big)\big\}}
{1-\exp\big\{-ih\big(\sum_{p \in \pp_h}k_p\log p+\sum_{m \in \no}l_m \log(m+\alpha)\big)\big\}}.
\end{eqnarray}
Consequently, from \eqref{rk-fo-5-5} and \eqref{rk-fo-8}, we have
$$
\lim\limits_{N \to \infty}g_{Nh}(\underline{k},\underline{l})=
\begin{cases}
1, &\text{if} \quad (\underline{k},\underline{l})=(\underline{0},\underline{0}),\cr
0, &\text{otherwise}.
\end{cases}
$$

Therefore, applying the continuity theorem on probability measures on compact groups (see \cite{HH-1977}), we obtain the assertion of the lemma.
\end{proof}

Suppose that $\sigma_1>\frac{1}{2}$ is fixed, and let
$$
v_1(m,n)=\exp\bigg\{-\bigg(\frac{m}{n}\bigg)^{\sigma_1}\bigg\} \quad \text{for} \quad m,n \in \nn,
$$
and
$$
v_2(m,n,\alpha)=\exp\bigg\{-\bigg(\frac{m+\alpha}{n+\alpha}\bigg)^{\sigma_1}\bigg\} \quad \text{for} \quad m \in \no, \quad n \in \nn.
$$
Define, for $n\in\nn$ and $\widehat{\omega}_h=(\ho_{1h},\ho_2)\in \Omega_h$,
\begin{eqnarray*}
\varphi_{h,n}(s)&=&\sum_{k \in \nn_m}\frac{c_k v_1(k,n)}{k^s} \quad \text{and} \quad
\zeta_n(s,\alpha;{\gb})=
\sum_{m \in \no}\frac{b_m v_2(m,n,\alpha)}{(m+\alpha)^s},
\end{eqnarray*}
\begin{eqnarray*}
\varphi_{h,n}(s,\ho_{1h})=
\sum_{k \in \nn_m}\frac{c_k\ho_{1h}(k) v_1(k,n)}{k^s}
\end{eqnarray*}
and
\begin{eqnarray*}
\zeta_n(s,\alpha, {\widehat{\omega}}_2;\gb)=
\sum_{m \in \no}\frac{b_m {\widehat{\omega}}_2(m) v_2(m,n,\alpha)}{(m+\alpha)^s}.
\end{eqnarray*}
These series are absolutely convergent for $\sigma>\frac{1}{2}$ (for the comments, see \cite{RK-KM-2015}). Now we will consider the weak convergence of the measures
$$
P_{Nh,n}(A):=\frac{1}{N+1}\# \bigg\{0 \leq k \leq N: \uZ_{h,n}(\us+ikh)\in A\bigg\}
$$
and
$$
{\widehat P}_{Nh,n}(A):=\frac{1}{N+1}\# \bigg\{0 \leq k \leq N: \uZ_{h,n}(\us+ikh, \ho_h)\in A\bigg\}
$$
for $A \in {\mathcal B}(\uH)$, where
$$
\underline{Z}_{h,n}(\us)=
(\vf_{h,n}(s_1),\zeta_n(s_2,\alpha;\gb))
$$
and
$$
\underline{Z}_{h,n}(\us,\ho_h)=
(\vf_{h,n}(s_1,\ho_{1h}),
\zeta_n(s_2,\alpha,\ho_2;\gb)).
$$

\begin{lemma}\label{rk-le-2}
Suppose that the conditions of Theorem~\ref{rk-th-4} hold. Then $P_{Nh,n}$ and ${\widehat P}_{Nh,n}$ both converge weakly to a same probability measure $P_n$ on
$(\uH,\mathcal{B}(\uH))$ as $N\to\infty$.
\end{lemma}

\begin{proof}
Because of the absolute convergence of the series for $\varphi_{h,n}(s)$, $\varphi_{h,n}(s,\widehat{\omega}_{1h})$, $\zeta_n(s,\alpha;{\gb})$ and $\zeta_n(s,\alpha,\widehat{\omega}_2;{\gb})$,
we can use Lemma~\ref{rk-le-1}, and Theorem~5.1 in \cite{PB-1968}, and argue in a way similar to the proof of Lemma~3.2 in \cite{RK-KM-2015}, to obtain the statement of the lemma.
\end{proof}

Now we need to pass from $\uZ_{h,n}(\us)$ to $\uZ_h(\us)$ and from $\uZ_{h,n}(\us,\omega_h)$ to $\uZ_h(\us,\omega_h)$, respectively. This can be done by using the approximation method together
with Lemma~\ref{rk-le-2}. For this purpose, we introduce a metric on $\uH$.

It is known that, for any open region $G$, there exists a sequence of compact sets $\{K_l: l\in \nn\}\subset {G}$ such that $G=\bigcup\limits_{l=1}^\infty K_l$, $K_l \subset K_{l+1}$ for all $l \in \nn$,
and, if $K$ is a compact set, then $K \subset K_l$ for some $l \in \nn$. For the functions $g_1,g_2 \in H(G)$, let
$$
\varrho_G(g_1,g_2)=\sum_{l=1}^{\infty}{2^{-l}}\frac{\sup_{s \in K_l}|g_1(s)-g_2(s)|}{1+\sup_{s \in K_l}|g_1(s)-g_2(s)|}.
$$
Set $\varrho_1=\varrho_{D_1}$ and $\varrho_2 =\varrho_{D_2}$. For
$\underline{g}_1=(g_{11},g_{21})$ and $\underline{g}_2=(g_{12},g_{22})\in \uH$,
$$
{\underline \varrho}(\underline{g}_1,\underline{g}_2)=\max\big\{\varrho_1(g_{11},g_{12}),\varrho_2(g_{21},g_{22})\big\}.
$$
Then ${\underline \varrho}(\underline{g}_1,\underline{g}_2)$ is a metric on the space $\uH$  which induces its topology of uniform convergence on compacta.


\begin{lemma}\label{rk-le-3}
Under conditions of Theorem~\ref{rk-th-4}, the  following relations hold:
\begin{equation}\label{rk-fo-9}
\lim_{n\to\infty}\limsup_{N\to\infty}\frac{1}{N+1}\sum_{k=0}^{N}
\underline{\rho}\big(\uZ_h(\us+i kh),
\uZ_{h,n}(\us+ikh)\big)=0
\end{equation}
and, for almost all $\omega_h\in\Omega_h$,
\begin{equation}\label{rk-fo-10}
\lim_{n\to\infty}\limsup_{N\to\infty}\frac{1}{N+1}\sum_{k=0}^{N}
\underline{\rho}\big(\uZ_h(\us+ikh,\omega_h),
\uZ_{h,n}(\us+ikh,\omega_h)\big)=0.
\end{equation}
\end{lemma}

To prove this lemma, we need some elements from ergodic theory.

Let $f_h=\big\{(p^{-ih}: p \in \pp_h), \big((m+\alpha)^{-ih}: m \in \no\big)\big\}$, and, on the pro\-ba\-bi\-li\-ty space $(\Omega_h,{\mathcal B}(\Omega_h),m^h_H)$,
define the measurable measure-preserving transformation $\Phi_h: \Omega_h \to \Omega_h$ by the formula
$\Phi_h(\omega_h)=f_h\omega_h$ for $\omega_h \in \Omega_h$.    Recall that a set $A \in {\mathcal B}(\Omega_h)$ is called invariant with respect to $\Phi_h$ if the sets $A$ and $\Phi_h(A)$ differ only by a set of zero $m_H^h$-measure, and
the transformation $\Phi_h$ is ergodic if its $\sigma$-field of invariant sets consists only of sets having $m_H^h$-measure equal to 0 or 1.

\begin{lemma}\label{rk-le-5}
	Suppose that $\alpha$ is transcendental, $h>0$, and $\exp\big\{\frac{2 \pi}{h}\big\}$ is rational. Then the transformation $\Phi_h$ is ergodic.
\end{lemma}

In the proof of this lemma, again the conditions for $\alpha$ and $\exp\big\{\frac{2 \pi}{h}\big\}$ play the essential role.    Therefore we give the full details.

\begin{proof}
	Let $\chi$ be a non-trivial character of $\Omega_h$. In the proof of Lemma~\ref{rk-le-1}, we have already known that such characters are given by
	$$
	\chi(\omega_h)=\prod_{p \in \pp_h}\omega_{1h}^{k_p}(p)\prod_{m \in \no}\omega_2^{l_m}(m)
	$$
	with only a finite number of integers $k_p$ and $l_m$ distinct from zero, where
	$\omega_h=(\omega_{1h},\omega_2)$, $\omega_{1h} \in \Omega_{1h}$, $\omega_2 \in \Omega_2$.
	Therefore,
	$$
	\chi(f_h)=\exp\bigg\{-ih\bigg(\sum_{p \in \pp_h}k_p\log p+\sum_{m \in \no}l_m \log(m+\alpha)\bigg)\bigg\}.
	$$
	As in the proof of Lemma~\ref{rk-le-1} (see the equation \eqref{rk-fo-6}), under the assumptions for $\alpha$ and $\exp\{\frac{2 \pi}{h}\}$, we have that
	\begin{equation}\label{rk-fo-11}
	\chi(f_h)\not =1
	\end{equation}
	for  $(\underline{k},\underline{l})\not =(\underline 0, \underline 0)$.
	
	Denote by $\ii_A$ the indicator function of the set $A$, and by ${\widehat \ii}_A(\chi)$  its Fourier transformation. Let  $A \in {\mathcal B}(\Omega_h)$ be an invariant set of the transformation $\Phi_h$. Then we have that
$\ii_{A}(f_h\omega_h)=\ii_A(\omega_h)$
for almost all $\omega_h \in \Omega_h$.
Therefore we obtain that
	\begin{eqnarray*}
		{\widehat \ii}_A(\chi)&=&\int_{\Omega_h}\chi(\omega_h)\ii_A(\omega_h)m_H^h (d \omega_h)
		=\int_{\Omega_h}\chi(f_h\omega_h)\ii_A(f_h\omega_h)m_H^h (d \omega_h)\cr
		&=&\chi(f_h)\int_{\Omega_h}\chi(\omega_h)\ii_A(\omega_h)m_H^h (d \omega_h)= \chi(f_h){\widehat \ii}_A(\chi).
	\end{eqnarray*}
	Hence, in view of \eqref{rk-fo-11}, for the non-trivial character $\chi$, we have
${\widehat \ii}_A(\chi)=0$.
	
	Now let $\chi_0$ be the trivial character of $\Omega_h$, i.e., $\chi_0(\omega_h)=1$ for all $\omega_h \in \Omega_h$. Suppose that ${\widehat \ii}_A(\chi_0)=u$. Taking into account the equalities
	$$
	\int_{\Omega_h}\chi(\omega_h)m_H^h(d \omega_h)=\begin{cases}
	1 &\quad \text{if} \quad \chi=\chi_0,\\
	0 &\quad \text{if} \quad \chi\not = \chi_0,
	\end{cases}
	$$
	and the fact ${\widehat \ii}_A(\chi)=0$ for $\chi\not =\chi_0$, we obtain that, for every character $\chi$ of $\Omega_h$,
	$$
	{\widehat \ii}_A(\chi)=u\int_{\Omega_h}\chi(\omega_h)m_H^h (d \omega_h)={\widehat u}(\chi).
	$$
	Since the function $\ii_A(\omega_h)$ is determined by its Fourier transform ${\widehat \ii}_A(\chi)$, it follows that $\ii_A(\omega_h)=u$ for almost all $\omega_h \in \Omega_h$. Moreover $u=0$ or $u=1$,
	because $\ii_A(\omega_h)$ is defined as the indicator of the set $A$. Thus $\ii_A(\omega_h)=0$ or $\ii_A(\omega_h)=1$ for almost all $\omega_h \in \Omega_h$. From this we find that $m_H^h(A)=0$ or $m_H^h(A)=1$,
	and therefore the transformation $\Phi_h$ is ergodic.
\end{proof}

\begin{proof}[Proof of Lemma~\ref{rk-le-3}]
Since the properties of the function $\vph(s)$ are similar to those of $\vf(s)$, in view of Lemma~\ref{rk-le-5}, the proof of the lemma goes in the same way as in the proof of Lemma~3 in \cite{RK-KM-2017-Pal}.
\end{proof}

On $(\uH,{\mathcal B}(\uH))$, for $A\in\mathcal{B}(\uH)$, define one more probability measure
$\hP_{Nh}$ by
$$
\hP_{Nh}(A)=\frac{1}{N+1}\#
\bigg\{0 \leq k \leq N: \uZ_h(\us+i kh,\omega_h)\in A\bigg\}.
$$

\begin{lemma}\label{rk-le-4}
Suppose that the conditions of Theorem~\ref{rk-th-4} are fulfilled. Then, on $(\uH,{\mathcal B}(\uH))$, there exists a probability measure $P_h$
such that the measures $P_{Nh}$ and $\hP_{Nh}$ both converge weakly to $P_h$ as $N\to\infty$.
\end{lemma}

\begin{proof}
The lemma can be proved similarly as Lemma~4 in \cite{RK-KM-2017-Pal}.
\end{proof}

\begin{proof}[Proof of Theorem~\ref{rk-th-4}]
Taking into account Lemma~\ref{rk-le-4}, the proof of Theo\-rem~\ref{rk-th-4} will be completed if we can show that the limit measure $P_h$ coincides with the measure $P_{\uZ h}$.    We can do this, using Lemma~\ref{rk-le-5} together with the Birkhoff-Khintchine ergodicity theo\-rem, in a standard way (for the details, see \cite{AL-1996} or \cite{JSt-2007}) .
\end{proof}

\section{Proof of Theorem~\ref{rk-th-3}}\label{sec-3}

The deduction of Theorem \ref{rk-th-3} from Theorem \ref{rk-th-4} is now rather standard.
We first need the explicit form of the support of the measure $P_{\uZ h}$. Recall that the support of $P_{\uZ h}$ is a minimal closed set $S_{\uZ h} \subset \uH$
such that $P_{\uZ h}(S_{\uZ h})=1$. Also we remind the arguments from \cite{RK-KM-2015} (or Section~2 in \cite{RK-KM-2017-Pal}).

Let $\varphi\in\widetilde{S}$ (and so $\varphi_h\in\widetilde{S}$), and $K_1$, $K_2$, $f_1$ and $f_2$ be as in the statement of Theorem~\ref{rk-th-3}.
We can find a real number $\sigma_0$ such that $\sigma^*<\sigma_0<1$ and a positive
number $M>0$, such that $K_1$ is included in the open rectangle
$$
D_M=\{s \in \cc: \;\sigma_0<\sigma<1,\; |t|<M\}.
$$
Since $\varphi(s)\in\st$, its pole is at most at $s=1$ (as for the function $\vf(s)$), then we find that
$D_{\varphi}=\{s \in \cc :\;\sigma>\sigma_0, \;\sigma\neq 1\}$.
Therefore $D_M$ is an open subset of $D_{\varphi}$.
Also we can find $T>0$ such that $K_2$ is included in the open rectangle
$$
D_T=\bigg\{s \in \cc: \;\frac{1}{2}<\sigma<1, \;|t|<T\bigg\}.
$$

Now we choose $D_1=D_M$ and $D_2=D_T$ in Theorem~\ref{rk-th-4}. Denote by $S_{\varphi}$ the set of all functions $f\in H(D_M)$ non-vanishing on $D_M$, or
constantly equivalent to $0$ on $D_M$.

\begin{theorem}\label{rk-th-5}
Suppose that $\alpha$ is transcendental, and, for $h>0$, $\exp\big\{\frac{2 \pi}{h}\big\}$ is rational. The support of the measure $P_{\underline{Z} h}$ is the set
$S_{\uZ h}=S_{\varphi}\times H(D_T)$.
\end{theorem}

\begin{proof}
This is an analogue of Theorem~5 from \cite{RK-KM-2017-Pal}.
\end{proof}


For completion of the proof  Theorem~\ref{rk-th-3}, we use the following well-known Mergelyan theorem  on the approximation of analytic functions by polynomials (see \cite{SNM-1952}).

\begin{lemma}\label{rk-le-6}
Let $K \subset \cc$ be a compact subset with connected complement, and $f(s)$ be a continuous function on $K$ which is analytic in the interior of $K$. Then, for every $\varepsilon>0$, there exists a polynomial $p(s)$ such that
$$
\sup\limits_{s \in K}|f(s)-p(s)|<\varepsilon.
$$
\end{lemma}

\begin{proof}[Proof of Theorem~\ref{rk-th-3}]
This is analogous to Section~4 in  \cite{RK-KM-2017-Pal}.
Since the function $f_1(s) \not = 0$ on $K_1$, by Lemma~\ref{rk-le-6}, there exists the polynomials $p_1(s)$ and $p_2(s)$ such that
\begin{equation}\label{rk-fo-14}
\sup\limits_{s \in K_1}\big|f_1(s)-e^{p_1(s)}\big|<\frac{\ve}{2} \quad \text{and} \quad
\sup\limits_{s \in K_2}\big|f_2(s)-p_2(s)\big|<\frac{\varepsilon}{2}.
\end{equation}

Next we define the set
$$
G=\bigg\{
(g_1,g_2)\in \uH: \sup\limits_{s\in K_1}|g_1(s)-e^{p_1(s)}|<\frac{\varepsilon}{2}, \
\sup\limits_{s\in K_2}|g_2(s)-p_2(s)|<\frac{\varepsilon}{2}
\bigg\},
$$
which is an open subset of the space $\uH$, and by Theorem~\ref{rk-th-5}, it is an open neighbourhood of the element $(e^{p_1(s)},p_2(s))$ of the support of $P_{\uZ h}$.
Therefore $P_{\uZ h}(G)>0$. Moreover, Theorem~\ref{rk-th-4} and an equivalent statement of the weak convergence in terms of open sets (see \cite{PB-1968}) together with the
definitions of $P_{N h}$ and $G$ show that
\begin{eqnarray*}\label{rk-fo-15}
\liminf\limits_{N \to \infty}\frac{1}{N+1}\#
\bigg\{
0 \leq k \leq N:  \uZ_h(\us+ikh) \in G
\bigg\} \geq P_{\uZ h}(G)>0
\end{eqnarray*}
or
\begin{eqnarray*}\label{rk-fo-16}
\liminf\limits_{N \to \infty}\frac{1}{N+1}\#
\bigg\{
0 \leq k \leq N: && \sup\limits_{s\in K_1}\big|\varphi_h(s+ikh)-e^{p_1(s)}\big|<\frac{\varepsilon}{2},\cr
&& \sup\limits_{s\in K_2}\big|\zeta(s+ikh,\alpha;\gb)-p_2(s)\big|<\frac{\varepsilon}{2}
\bigg\}>0.
\end{eqnarray*}
From the last inequality and \eqref{rk-fo-14}
we obtain the assertion of Theorem~\ref{rk-th-3}.
\end{proof}

\section{Proof of Theorem~\ref{rk-th-2-1}}\label{sec-4}

For the proof of Theorem~\ref{rk-th-2-1}, we adopt the well-known probabilistic method (see, for example, \cite{RK-KM-2017-BAMS}, \cite{AL-DS-2012}), based on the joint limit theorems for probability measures in the space of holomorphic functions.

Let $\varphi\in\mathcal{M}$, $D_1\subset D_{\varphi}$, $D_2\subset D_{\zeta}$, and
$$
\uH_r=H(D_1)\times \underbrace{H(D_2)\times ...\times H(D_2)}\limits_{r}.
$$
Moreover, define the torus
$
\uO_r=\Omega_1\times \Omega_{2 1}\times ... \times \Omega_{2 r}
$
where
$\Omega_{2j}=\Omega_2$ for all $j=1,...,r$.  Then $\uO_r$ is a compact topological Abelian group, and we obtain a new probability space $(\uO_r,{\mathcal B}(\uO_r), m_{Hr})$ with the Haar measure $m_{Hr}$.

Suppose that $\uom_r:=(\omega_1,\omega_{2 1},...\omega_{2 r})$ is an element of $\uO_r$, and, for brevity, let
$\us_r=(s_1,s_{21},...,s_{2r}) \in \cc^{r+1}$,  $\uA=(\alpha_1,...,\alpha_r)$ and $\ugb=(\gb_1,...,\gb_r)$, where $\gb_j=\{b_{mj}:m\in\mathbb{N}_0\}$, $j=1,...,r$,  is a periodic sequence of complex numbers with the minimal period $k_j$.  By $\zeta(s,\alpha_j;\gb_j)$ we denote the corresponding periodic Hurwitz zeta-function, $j=1,...,r$.
Now, on the probability space $(\uO_r,{\mathcal B}(\uO_r), m_{Hr})$, define an $\uH_r$-valued random element $\uZ_r(\us_r,\uA,\uom_r;\ugb)$ by the formula
$$
\uZ_r(\us_r,\uA,\uom_r;\ugb)=\big(\varphi(s_1,\omega_1), \zeta(s_{2 1},\alpha_1,\omega_{2 1};\gb_1),..., \zeta(s_{2 r},\alpha_r,\omega_{2 r};\gb_r)\big).
$$
Here
\begin{equation}\label{rk-fo-17}
\varphi(s_1,\omega_1)=\sum_{k=1}^{\infty}\frac{c_k \omega_1(k)}{k^{s_1}}, \quad s_1 \in D_1,
\end{equation}
and
\begin{equation}\label{rk-fo-18}
\zeta(s_{2j},\alpha_j,\omega_{2j};\gb_j)=\sum_{m=0}^{\infty}
\frac{b_{mj}\omega_{2j}(m)}{(m+\alpha_j)^{s_{2j}}}, \quad s_{2j}\in D_2, \quad j=1,...r.
\end{equation}
Denote by $P_{\uZ r}$ the distribution of the random element  $\uZ_r(\us_r,\uA,\uom_r;\ugb)$ defined by
$$
P_{\uZ r}(A):=m_{Hr}\{\uom_r\in\uO_r : \;  \uZ_r(\us_r,\uA,\uom_r;\ugb) \in A\},\quad A \in {\mathcal B}(\uH_r).
$$

Now, on $(\uH_r,{\mathcal B}(\uH_r))$, we define a probability measure $P_{Nr}$ by
$$
P_{Nr}(A)=\frac{1}{N+1}\#\big\{0 \leq k \leq N: \;
\uZ_r(\us_r+ikh)
\in A\big\}, \quad A\in\mathcal{B}(\underline{H}),
$$
where $\us_r+ikh=(s_1+ikh,s_{21}+ikh,...,s_{2r}+ikh)$ and
$$
\uZ_r(\us_r)= (\varphi(s_1),\zeta(s_{21},\alpha_1;\gb_1),...,\zeta(s_{2r},\alpha_r;\gb_r)).
$$

Then we have following mixed discrete joint limit theorem.

\begin{theorem}\label{rk-th-7}
Suppose $\varphi\in\mathcal{M}$.   If the elements of the set $L(\pp,\uA,h)$ are linearly independent over $\qq$, then  $P_{Nr}$ converges weakly to $P_{\uZ r}$ as $N \to \infty$.
\end{theorem}


The proof of this theorem goes along the way analogous to that of Lemma~5.1 in \cite{RK-KM-2017-BAMS}: first we prove a discrete joint limit theorem on the torus $\uO_r$ (Lemma \ref{rk-le-7} below), then, after showing a discrete joint limit theorem for absolutely convergent series, we approximate in the mean, and finally, using ergodic theory, we obtain an explicit form of the measure $P_{Nr}$, i.e., we show that it converges weakly to the distribution of the random element $\uZ_r(\us_r,\uA,\uom_r;\ugb)$.
As in the proof of Theorem~\ref{rk-th-4}, the main point of the proof of Theorem \ref{rk-th-7} is the first step (Lemma \ref{rk-le-7}).    Therefore we only describe the proof of this
lemma, and omit the other part of the proof of Theorem \ref{rk-th-7}.

Define, for $A \in {\mathcal B}(\uO_r)$,
$$
Q_{Nr}(A):=\frac{1}{N+1}\# \bigg\{0 \leq k \leq N: \bigg(\big(p^{-ikh}: p \in \pp\big), \big((m+\alpha_j)^{-ikh}: m \in \no, j=1,...,r\big)\bigg)\in A\bigg\}.
$$

\begin{lemma}\label{rk-le-7}
Suppose that the elements of the set $L(\pp,\uA,h)$ are linearly independent over $\qq$. Then $Q_{Nr}$ converges weakly to the Haar measure $m_{Hr}$ as $N \to \infty$.
\end{lemma}

\begin{proof}
As in the proof of Lemma~\ref{rk-le-1}, here we use the Fourier transform method.
The Fourier transform of the measure $Q_{Nr}$ is given by
$$
g_{Nr}(\underline{k},\underline{l}_1,...,\underline{l}_r)=\int_{\uO_r}\bigg(\prod_{p \in \pp}\omega_1^{k_p}(p)\prod_{j=1}^{r}\prod_{m \in \no}\omega_{2j}^{l_{mj}}(m)\bigg)\d Q_{Nr},
$$
where only a finite number of integers $k_p$ and $l_{mj}$ are distinct from zero.
Thus, from the definition of $Q_{Nr}$, we have
\begin{eqnarray}\label{rk-fo-19-0}
&& g_{Nr}(\underline{k},\underline{l}_1,...,\underline{l}_r)\cr
&&\quad =
\frac{1}{N+1}\sum_{k=0}^{N}\prod_{p \in \pp}p^{-ikk_ph}\prod_{j=1}^{r}\prod_{m \in \no}(m+\alpha_j)^{-ikl_{mj}h}\cr
&&\quad =
\frac{1}{N+1}\sum_{k=0}^{N}
\exp\bigg\{-ikh\bigg(\sum_{p \in \pp}k_p\log p+\sum_{j=1}^{r}\sum_{m \in \no}l_{mj} \log(m+\alpha_j)\bigg)\bigg\}.
\end{eqnarray}
By the assumption of the lemma, the elements of the set $L(\pp,\uA,h)$ are linearly independent over $\qq$. Then
$$
\sum_{p \in \pp}k_p\log p+\sum_{j=1}^{r}\sum_{m \in \no}l_{mj} \log(m+\alpha_j)=0
$$
if  $\underline{k}=\underline l_1=...=\underline l_r=\underline{0}$, and in this case
 $g_{Nr}(\underline{k},\underline{l}_1,...,\underline{l}_r)=1$.

On the other hand, for $(\underline{k},\underline l_1,...,\underline l_r)$ with at least one non-zero vector $\underline k$, $\underline l_r$, $j=1,...,r$,
\begin{equation}\label{rk-fo-19-2}
\exp\bigg\{-ih\bigg(\sum_{p \in \pp}k_p\log p+\sum_{j=1}^{r}\sum_{m \in \no}l_{mj}
 \log(m+\alpha_j)\bigg)\bigg\}\not =1.
\end{equation}
In fact, if \eqref{rk-fo-19-2} is false, then, for  some integer $a\not= 0$,
\begin{equation}\label{rk-fo-19-3}
\exp\bigg\{-ih
\bigg(
\sum_{p \in \pp}k_p\log p+\sum_{j=1}^{r}  \sum_{m \in \no} l_{mj} \log(m+\alpha_j)
\bigg)\bigg\}
 =\exp\{2 \pi i a\},
\end{equation}
and
$$
\sum_{p \in \pp}k_p\log p+\sum_{j=1}^{r}  \sum_{m \in \no} l_{mj} \log(m+\alpha_j)=
-\frac{2 \pi a}{h},
$$
but this contradicts to the linear independence of the elements of the set $L(\pp,\uA,h)$.
Therefore we now obtain
$$
\lim\limits_{N \to \infty} g_{Nr}(\underline{k},\underline{l}_1,...,\underline{l}_r)=
\begin{cases}
1, &\text{if} \quad (\underline{k},\underline{l}_,...,\underline{l}_r)=(\underline{0},\underline{0},...,\underline{0}),\cr
0, &\text{otherwise}.
\end{cases}
$$
This together with the continuity theorem for probability measures on compact groups gives the statement of the lemma.
\end{proof}

Now assume $\varphi\in\widetilde{S}$, and $K_1$, $K_{2j}$, $f_1(s)$, $f_{2j}(s)$, $j=1,...,r$, are as in the statement of Theorem~\ref{rk-th-2-1}.
We choose $D_1=D_M$ and $D_2=D_T$, where $K_1\subset D_M$ (as in Section \ref{sec-3})
and $K_{2j}\subset D_T$ for all $j=1,\ldots,r$.
As a special case of Lemma~5.2 from \cite{RK-KM-2017-BAMS} (when all $l(j)=1$, $j=1,...,r$, i.e., $\lambda=r$), we have

\begin{lemma}\label{rk-le-8}
The support of the measure $P_{\uZ r}$ is the set $S_\vf \times H(D_T)^r$.
\end{lemma}

Next, using Mergelyan's theorem (Lemma~\ref{rk-le-6}), we can find polynomials $p_1(s)$ and $p_{2j}(s)$, $j=1,...,r$,  which approximate $\log f_1(s)$ and $f_{2j}(s)$, $j=1,...,r$, respectively.
Moreover
from Theorem~\ref{rk-th-7} and Lemma~\ref{rk-le-8}, we have that
\begin{align*}\label{rk-fo-19}
		\liminf\limits_{N \to \infty}
		\frac{1}{N+1}
		\#
		\bigg\{0\leq k \leq N: \
		& \sup\limits_{s \in K_1}|\varphi(s+ikh)-e^{p_1(s)}|<\varepsilon, \cr &  \sup\limits_{1\leq j \leq r}\sup\limits_{s\in K_{2j}}|\zeta(s+ikh,\alpha_j;\gb_j)-p_{2j}(s)|<\varepsilon\bigg\}>0.
\end{align*}
This completes the proof of Theorem \ref{rk-th-2-1}.

\section{Proof of Theorem~\ref{rk-th-3-1}}\label{sec-5}

Theorem~\ref{rk-th-3-1} can be shown in a  way similar to Theorem~\ref{rk-th-2-1}. Therefore as in the preceding section we only explain the points at which conditions of Theorem~\ref{rk-th-3-1} play an essential role.


First, on $(\uH_r,{\mathcal B}(\uH_r))$, we consider
$$
P_{Nhr}(A)=\frac{1}{N+1}\#\big\{0 \leq k \leq N: \;
\uZ_{hr}(\us_r+ikh)
\in A\big\}, \quad A\in\mathcal{B}(\uH_r),
$$
where $\us_r+ikh$ is as in Section~\ref{sec-4} and
$$
\uZ_{hr}(\us_r)= (\varphi_h(s_1),\zeta(s_{21},\alpha_1;\gb_1),...,\zeta(s_{2r},\alpha_r;\gb_r)).
$$

The torus
$
\uO_{hr}=\Omega_{1h}\times \Omega_{2 1}\times ... \times \Omega_{2 r}
$
is a compact topological Abelian group, which leads to the new probability space $(\uO_{hr},{\mathcal B}(\uO_{hr}), m_{Hr}^h)$ with the Haar measure $m_{Hr}^h$.
The ele\-ments  $\omega_{1h}(p)$ and $\omega_{2j}(m)$, $j=1,...,r$, are of the same meaning as in Sections~\ref{sec-2} and \ref{sec-4}, respectively.

Let  $\uom_{hr}:=(\omega_{1h},\omega_{2 1},...\omega_{2 r})$ be an element of $\uO_{hr}$, and, on $(\uO_{hr},{\mathcal B}(\uO_{hr}), m_{Hr}^h)$,  define an $\uH_{r}$-valued random element $\uZ_{hr}(\us_r,\uA,\uom_{hr};\ugb)$ by the formula
$$
\uZ_{hr}(\us_r,\uA,\uom_{hr};\ugb):=\big(\varphi_h(s_1,\omega_{1h}), \zeta(s_{2 1},\alpha_1,\omega_{2 1};\gb_1),..., \zeta(s_{2 r},\alpha_r,\omega_{2 r};\gb_r)\big),
$$
where $\varphi_h(s_1,\omega_{1h})$ and $\zeta(s_{2j},\alpha_j,\omega_{2j};\gb_j)$, $j=1,...,r$, are defined by \eqref{rk-fo-5-1} and \eqref{rk-fo-18}, respectively. Let $P_{\uZ hr}$ denote its distribution, i.e., $P_{\uZ hr}$ is the probability measure on $(\uH_r,{\mathcal B}(\uH_r))$ defined by
$$
P_{\uZ hr}(A):=m_{Hr}^h\big\{\uom_{hr}\in\uO_{hr} : \;  \uZ_{hr}(\us_r,\uA,\uom_{hr};\ugb) \in A\big\},\quad A \in {\mathcal B}(\uH_r).
$$

Then, as a generalization of Theorem \ref{rk-th-4}, we can show the following mixed discrete joint limit theorem in the sense of weakly convergent probability measures in the space of holomorphic functions.

\begin{theorem}\label{rk-th-10}
Let $\varphi\in\mathcal{M}$ (and so $\varphi_h\in\mathcal{M}$).
Suppose that, for $h>0$, $\exp\big\{\frac{2 \pi}{h}\big\}$ is a rational number. Let the numbers $\alpha_1,...,\alpha_r$ be algebraically independent over $\qq$. Then  $P_{Nhr}$ converges weakly to $P_{\uZ h r}$ as $N \to \infty$.
\end{theorem}


In the proof of this theorem, a crucial role is played by the following mixed discrete joint limit theorem on the torus $\uO_{hr}$.

\begin{lemma}\label{rk-le-10}
Suppose that $\vph(s)$, $\exp\big\{\frac{2 \pi}{h}\big\}$ and $\alpha_j$, $j=1,...,r$ are as in Theorem~\ref{rk-th-10}.
Then the probability measure
$$
Q_{Nhr}(A):=\frac{1}{N+1}\# \bigg\{0 \leq k \leq N: \bigg(\big(p^{-ikh}: p \in \pp_h\big), \big((m+\alpha_j)^{-ikh}: m \in \no, j=1,...,r\big)\bigg)\in A\bigg\},
$$
$A \in {\mathcal B}(\uO_{hr})$,  converges weakly to the Haar measure $m_{Hr}^h$ as $N \to \infty$.
\end{lemma}

\begin{proof}
The Fourier transform $g_{Nhr}$ of the measure $Q_{Nhr}$ is defined by the formula
$$
g_{Nhr}(\underline{k},\underline{l}_1,...,\underline{l}_r)=\int_{\uO_{hr}}\bigg(\prod_{p \in \pp_h}\omega_{1h}^{k_p}(p)\prod_{j=1}^{r}\prod_{m \in \no}\omega_{2j}^{l_{mj}}(m)\bigg)\d Q_{Nhr},
$$
where only a finite number of integers $k_p$ and $l_{jm}$ are distinct from zero.
Similar to \eqref{rk-fo-19-0}, we obtain
\begin{eqnarray}\label{rk-fo-20}
&& g_{Nhr}(\underline{k},\underline{l}_1,...,\underline{l}_r)\cr
&&\quad =
\frac{1}{N+1}\sum_{k=0}^{N}
\exp\bigg\{-ikh\bigg(\sum_{p \in \pp_h}k_p\log p+\sum_{j=1}^{r}\sum_{m \in \no}l_{mj} \log(m+\alpha_j)\bigg)\bigg\}.
\end{eqnarray}

From the condition on algebraic independence for $\alpha_1,...,\alpha_r$ it follows that they are trans\-cen\-den\-tal numbers (see \cite{YN-1996}).   Arguing similarly as in the proof of Lemma~\ref{rk-le-1}, we have
\begin{equation}\label{rk-fo-21}
	\exp\bigg\{-ih\bigg(\sum_{p \in \pp_h}k_p\log p+\sum_{j=1}^{r}\sum_{m \in \no}l_{mj}
	\log(m+\alpha_j)\bigg)\bigg\}\not =1
	\end{equation}
for $(\underline{k},\underline l_1,...,\underline l_r)$ when not all components are zero-vectors.
On the other hand, if $(\underline{k},\underline l_1,...,\underline l_r)=(\underline 0, \underline 0,...,\underline 0)$,
then the left-hand side of \eqref{rk-fo-21} is clearly equal to 1.
Thus we have
$$
\lim\limits_{N \to \infty} g_{Nhr}(\underline{k},\underline{l}_1,...,\underline{l}_r)=
	\begin{cases}
	1, &\text{if} \quad (\underline{k},\underline{l}_1,...,\underline{l}_r)=(\underline{0},\underline{0},...,\underline{0}),\cr
	0, &\text{otherwise}.
	\end{cases}
$$
This and the continuity theorem for probability measures on compact groups proves the lemma.
\end{proof}

The rest of the proof of Theorem~\ref{rk-th-3-1} goes in the same line as proof of Theorem~\ref{rk-th-2-1}, based on Theorem \ref{rk-th-10}, only with the minor changes replacing $\vf(s)$ by $\vph(s)$ together with $\pp$ by $\pp_h$, respectively.


\end{document}